\documentclass[reqno,10pt]{amsart}

\input xy
\xyoption{all}
\usepackage{graphicx}
\usepackage{transparent}
\usepackage{amsthm}
\usepackage{amssymb}
\usepackage{amsmath}
\usepackage{amscd}
\usepackage{amsopn}
\usepackage{url}
\usepackage{hyperref}\hypersetup{colorlinks}

\usepackage{color} %\textcolor{red}{Test}

\definecolor{darkred}{rgb}{1,0,0} %can change the intensity in [0,1]
\definecolor{darkgreen}{rgb}{0,0.8,0}
\definecolor{darkblue}{rgb}{0,0,1}

\hypersetup{colorlinks, linkcolor=darkblue, filecolor=darkgreen,
urlcolor=darkred, citecolor=darkgreen}

\numberwithin{equation}{section}
\newtheorem {Theorem}{Theorem}
\numberwithin{Theorem}{section}

\newtheorem {Lemma}[Theorem]    {Lemma}
\newtheorem {Claim}[Theorem]    {Claim}

\newtheorem {Proposition}[Theorem]{Proposition}

\theoremstyle{definition}
\newtheorem{Definition}[Theorem]{Definition}
\theoremstyle{remark}
\newtheorem{Remark}[Theorem]{Remark}
\newtheorem{Example}[Theorem]{Example}

\expandafter\chardef\csname pre amssym.def
at\endcsname=\the\catcode`\@ \catcode`\@=11
\def\undefine#1{\let#1\undefined}
\def\newsymbol#1#2#3#4#5{\let\next@\relax
 \ifnum#2=\@ne\let\next@\msafam@\else
 \ifnum#2=\tw@\let\next@\msbfam@\fi\fi
 \mathchardef#1="#3\next@#4#5}
\def\mathhexbox@#1#2#3{\relax
 \ifmmode\mathpalette{}{\m@th\mathchar"#1#2#3}%
 \else\leavevmode\hbox{$\m@th\mathchar"#1#2#3$}\fi}
\def\hexnumber@#1{\ifcase#1 0\or 1\or 2\or 3\or 4\or 5\or 6\or 7\or 8\or
 9\or A\or B\or C\or D\or E\or F\fi}

\font\teneufm=eufm10 \font\seveneufm=eufm7 \font\fiveeufm=eufm5
\newfam\eufmfam
\textfont\eufmfam=\teneufm \scriptfont\eufmfam=\seveneufm
\scriptscriptfont\eufmfam=\fiveeufm

\catcode`\@=\csname pre amssym.def at\endcsname

\def    \eps    {\epsilon}

\newcommand{\FF}{{\mathcal F}}

\newcommand{\CA}{{\mathcal A}}

\newcommand{\CS}{{\mathcal S}}

\newcommand{\supp}{\operatorname{supp}}

\newcommand{\hook}{\hookrightarrow}

\newcommand{\hu}{\hat{u}}

\newcommand{\A}{{\mathcal A}}

\newcommand{\Cc}{{\mathcal C}}

\newcommand{\Ff}{{\mathcal F}}

\newcommand{\lmd}{{\lambda - \delta}}
\newcommand{\lpd}{{\lambda + \delta}}
\newcommand{\Klre}{{K_{\lambda,r,\varepsilon}}}
\newcommand{\Kle}{{K_{\lambda,\varepsilon}}}
\newcommand{\D}{{\Delta}}

\def    \C      {{\mathbb C}}

\def    \reals      {{\mathbb R}}
\def    \Z      {{\mathbb Z}}

\def    \12    {{\frac{1}{2}}}

\def    \bg    {\bar{\gamma}}
\def    \bc    {\bar{c}}

\def    \by    {\bar{y}}
\def    \p      {\partial}

\def    \Sp     {\operatorname{Sp}}

\def    \HF     {\operatorname{HF}}

\def    \bx     {\bar{x}}
\def    \by     {\bar{y}}

\def    \MUCZ  {\operatorname{\mu_{\scriptscriptstyle{CZ}}}}

\begin{document}

%%%%%%%%%%%%%%%%%%%%%%%%%%%%%%
%   TEXT FORMATTING

\setlength{\textheight}{7.5in} \setlength{\textwidth}{5in}
%\setlength{\smallskipamount}{6pt}
%\setlength{\medskipamount}{10pt}
%\setlength{\bigskipamount}{16pt}

%%%%%%%%%%%%%%%%%%%%%%%%%%

%%%%%%%%%%%           BEGINNING OF  TEXT

%%%%%%%%%%%%%%%%%%%%%%%%%%

\title[Rigidity of the Coisotropic Maslov Index]{\small On the Rigidity of the Coisotropic Maslov Index on Certain Rational Symplectic Manifolds}

\author[Marta Bator\'eo]{Marta Bator\'eo}

\address{Department of Mathematics, University of California at Santa Cruz,
Santa Cruz, CA 95064, USA} \email{mbator@ucsc.edu}

\begin{abstract}
We revisit the definition of the Maslov index of loops in
coisotropic submanifolds tangent to the characteristic foliation of
this submanifold. This Maslov index is given by the mean index of a
certain symplectic path which is a lift of the holonomy along the
loop. We prove a Maslov index rigidity result for stable coisotropic
submanifolds in a broad class of ambient symplectic manifolds.
Furthermore, we establish a nearby existence theorem for the same
class of ambient manifolds.
\end{abstract}

\maketitle

\tableofcontents

\section{Introduction and Main Results}

\subsection{Introduction}
The main result of this paper is the Maslov index and symplectic
area rigidity for coisotropic submanifolds in a broad class of
ambient symplectic manifolds. In~\cite{Zi} and~\cite{Gi:maslov}, the
Maslov index is defined for loops in coisotropic submanifolds which
are tangent to the characteristic foliation of the coisotropic
submanifold. The Maslov index of such a loop, $x\colon S^1
\rightarrow W,$ is the (Conley-Zehnder) mean index $\D$ of a
symplectic path which is a lift of the holonomy along the loop to
the pull-back bundle $x^*TW.$ Although such a lift is not unique,
the coisotropic Maslov index $\mu$ is well-defined. The Maslov index
is a real valued index and it generalizes the usual Lagrangian
Maslov index.

With this definition of the coisotropic Maslov index, we prove a
result on the Maslov class rigidity. More specifically, given a
closed displaceable stable coisotropic submanifold, we show that
there exists a non-trivial loop lying in the submanifold with Maslov
index bounded below by $1$ and above by $2n+1-k,$ where $2n$ is the
dimension of the symplectic manifold and $k$ the codimension of the
coisotropic submanifold. Moreover, the result gives bounds on the
symplectic area bounded by the loop; this area is positive and
bounded above by the displacement energy of the coisotropic
submanifold. This result was proved by Ginzburg in~\cite{Gi:maslov}
for ambient symplectic manifolds which are symplectically
aspherical. The case where the characteristic foliation is a
fibration is also considered in \cite{Zi}. In this paper, we extend
the result to certain rational manifolds which need not be
symplectically aspherical. In the \emph{spherical} case, the
obtained loop may be trivial with non-trivial capping. Hence, in our
theorem we state conditions on the ambient manifold for which this
loop is non-trivial and has the referred bounds on the Maslov index
and on the symplectic area. For instance, we have non-triviality and
the desired bounds when the manifold is negative monotone.

The Maslov class rigidity for Lagrangian submanifolds was originally
studied by Viterbo in~\cite{Vi} for the Lagrangian torus and by
Polterovich in~\cite{Po1, Po2}, for instance, for monotone
Lagrangian submanifolds. These results show that the Maslov class
satisfies certain restrictions. Namely, the minimal Maslov number
lies between $1$ and $n+1$. Audin was the first to conjecture (as
far as we know) that the minimal Maslov number is $2$ for the
Lagrangian torus; cf.~\cite{Au}. Fukaya proved this conjecture
in~\cite{Fu}. There are two methods to prove this type of results.
One approach, introduced by Gromov (see~\cite{Gr}), uses holomorphic
curves. This approach is the one used, for instance, by Audin and
Polterovich (see also~\cite{ALP}). A different approach relies on
Hamiltonian Floer homology and is found, for instance, in the work
of Viterbo, Kerman and Sirikci; see also~\cite{Ke:pinned, KS}.

The proof of our result follows the method used by Ginzburg
in~\cite{Gi:maslov} which is based on the second approach mentioned
above together with the stability condition and certain lower bounds
on the energy estimated by Bolle in~\cite{Bo:fr, Bo:en}. The proof
also relies on a suitable action selector (introduced
in~\cite{Ke:pinned, KS}).

Furthermore, we state a theorem (and outline its prove) of dense or
nearby existence, that is, a theorem which guarantees the existence
of periodic orbits for a dense set of energy levels. This result is
presented in~\cite{Gi:coiso} for symplectically aspherical manifolds
and as mentioned there it can be viewed as a generalization of the
existence of closed characteristics on stable hypersurfaces in
$\reals^{2n}$, established in~\cite{HZ}. We state this nearby
existence theorem for a broader class of rational symplectic
manifolds.

\subsection{Coisotropic Maslov Index}
\label{subsection:maslov_index} Let $(W^{2n},\omega)$ be a
symplectic manifold and $M^{2n-k}$ a closed coisotropic submanifold
of $W$ of codimension $k.$ Then $(T_pM)^{\omega} \subseteq$~$T_pM$
for each $p\in M$ and, denoting by $\omega_{M}$ the restriction of
$\omega$ to $M,$ we note that the distribution $TM^{\omega} := \ker
\omega_{M}$ on $M$ is integrable. By the Frobenius theorem, there is
a foliation $\Ff$ (the \emph{characteristic foliation}) on $M$ whose
tangent spaces are given by $TM^{\omega},$ i.e., $T\Ff = \ker
\omega_M$, and the rank of this foliation is $k$.

Consider a capped loop $\bx=(x,u)$ tangent to $T\Ff$ and the
holonomy along~$x$
$$
H_t\colon T^{\perp}\Ff_{x(0)} \rightarrow T^{\perp}\Ff_{x(t)}.
$$
There is a symplectic vector bundle decomposition of the restriction
of $TW$ to $M$:
$$
TW\big|_{M}= (T\Ff \oplus T^{\perp}M)\oplus T^{\perp}\Ff
$$
where we identify the normal bundle $T^{\perp}\Ff$ to $\Ff$ in $M$
with $TM/T\Ff$ and the normal bundle $T^{\perp}M$ to $M$ in $W$ with
$TW/TM.$ Lift the holonomy along $x$ to $x^* TW.$ The capping $u$
gives rise to a symplectic trivialization of $x^*TW$, unique up to
homotopy, and hence this lift can be viewed as a symplectic path
$$\Psi\colon [0,1] \rightarrow \Sp(2n).$$
Following \cite{Zi} (see also \cite{Gi:maslov}) we adopt
\begin{Definition}
The \emph{coisotropic Maslov index} is defined (up to a sign) as the
mean index of this path, i.e.,
$$
\mu(x,u) :=-\D(\Psi).
$$
\end{Definition}
This Maslov index is real valued (see
Example~\ref{example:realMaslovindex}) and is independent of the
lift of the holonomy along $x.$ However, in general, it depends on
the trivialization arising from the capping $u.$ We refer the reader
to the appendix (section~\ref{section:coiso_maslov_index}) for the
definitions of the indices. The proof that this Maslov index is
well-defined can be found in \cite{Zi}. In the appendix, for the
sake of completeness, we give a direct proof of this fact.
\begin{Example}
\label{example:realMaslovindex} Consider the Hamiltonian defined in
$(\mathbb{C}^n,\omega_0)$ by
$$ H(z):=1/2\sum_{l=1}^{n} \lambda_l |z_l|^2$$ with $\lambda_l\in
\reals^+$ (where $\omega_0$ is the standard symplectic form). The
ellipsoid defined as the regular level set $H^{-1}(\{1\})$ is a
hypersurface (and hence a coisotropic submanifold) of
$\mathbb{C}^n.$ For each $j=1,\ldots, n,$ the loop parameterized by
$$\gamma_j(t):=(0,\ldots,0,z_j(t),0,\ldots,0)$$
where
$$
z_j(t)=e^{- i\lambda_j t}z_j
$$
(with $|z_j|^2=2 / \lambda_j$ and $t\in[0,2\pi / \lambda_j]$) is a
periodic orbit of the Hamiltonian system of $H$ lying in
$H^{-1}(\{1\})$. A calculation shows that the Maslov index of the
loop $(\gamma_j,u_j)$ is given by
$$
\mu(\gamma_j,u_j)=-\D(\gamma_j,u_j)=\frac{2}{\lambda_j}
\displaystyle\sum_{l=1}^{n}\lambda_l
$$
where $u_j$ is some capping of $\gamma_j.$ In this case, the index
is independent of the capping we use.

To compute $\mu(\gamma_j,u_j),$ we use
$\Psi_t=d(\varphi^t_H)_{\gamma(0)}$ the linearized flow along
$\gamma.$ The foliation $\Ff$ is formed by the integral curves of
$\varphi^t_H.$ See section~\ref{subsubsection:ham_act__mean_ind} for
the description of the Maslov index when the loop is a periodic
orbit of a Hamiltonian.
\end{Example}

\subsection{Rigidity of the Coisotropic Maslov Index (Main Theorem)}
In this section, we state and discuss the main theorem of this
paper.
\begin{Theorem}
\label{theorem:main} Let $(W^{2n},\omega)$ be a rational closed
symplectic manifold, $M^{2n-k}$~$\subset$~$W^{2n}$ a closed stable
displaceable coisotropic submanifold of $W$ and $\Ff$ its
characteristic foliation.

Assume that one of the following conditions is satisfied
        \begin{itemize}
        \item $W$ is negative monotone,
        \item $e(M)< \hbar,$ where $e(M)$ is the displacement energy of $M$ and $\hbar$ is the rationality constant of $W,$
        \item $2n+1<2N$, where $N$ is the minimal Chern number of $W.$
        \end{itemize}

Then, for all $\varepsilon >0$, there exists a capped loop
$\bg=(\gamma, v)$ such that $\gamma$ is a non-trivial loop tangent
to $\Ff$ and
    \begin{eqnarray*}
    1 \leq &\mu(\bg)&  \leq 2n+1-k,\\
    0 <    &\mbox{Area}(\bg)& \leq e(M) + \varepsilon,
    \end{eqnarray*}
where $\mbox{Area}(\bg) :=\displaystyle\int_v \omega.$
\end{Theorem}

\begin{Remark}
The condition that $W$ is closed can be replaced in the theorem by
geometrically bounded and wide. Recall that a symplectic manifold is
said to be \emph{wide} if it admits an arbitrarily large, compactly
supported, autonomous Hamiltonian whose Hamiltonian flow has no
non-trivial contractible periodic orbits of period less than or
equal to one; see~\cite{Gu} for more details. The proof of the
theorem in this case is essentially the same as when $W$ is closed.
\end{Remark}

\begin{Remark} In~\cite{Gi:maslov}, Ginzburg proves Theorem~\ref{theorem:main} when $W$ is
symplectically aspherical; see
section~\ref{subsection:symplmnflds&hams} for the definition.
\end{Remark}

\begin{Remark} The requirements that $M$ is displaceable and stable are essential. For instance, a closed manifold
$M$ viewed as the zero section of its cotangent bundle $T^*M$ is not
displaceable (cf.~\cite{Gr}) and the Maslov index of a loop in $M$
is always trivial since $\pi_2(T^*M,M)=0$. Moreover, the assumption
that $M$ is stable cannot be entirely omitted: there exist
Hamiltonian systems having no periodic orbits on a compact energy
level which arise as counterexamples to the Seifert conjecture;
 cf.~\cite{Gi:ham_dyn_sys_w/o_per_orbs, GG:counterexample}.
\end{Remark}

\subsection{Acknowledgments} The author is grateful to Viktor
Ginzburg for posing the problem and useful discussions and Fabian
Ziltener for valuable remarks.

\section{Preliminaries}
In this section we introduce the notation used throughout the paper
and review some facts needed to prove the results.

\subsection{Symplectic Manifolds and Hamiltonians}
\label{subsection:symplmnflds&hams} Let $(W^{2n},\omega)$ be a
closed rational symplectic manifold and consider an almost complex
structure $J$ on $W$ compatible with $\omega,$ i.e., such that
$\left< \xi, \eta \right> := \omega(\xi, J\eta)$ is a Riemannian
metric on $W.$

Recall that $(W, \omega)$ is \emph{closed} if it is compact with no
boundary and is said to be \emph{(spherically) rational} if the
group
$$
\left<[\omega], \pi_2(W) \right> \subset \reals
$$
formed by the integrals of $\omega$ over the spheres in $W$ is
discrete, that is,
$$
\left< [\omega], \pi_2(W) \right>=\hbar \Z
$$
where $\hbar\geq 0.$ When $\left< [\omega], \pi_2(W) \right>=0$ we
set $\hbar=\infty.$ The constant $\hbar$ is called the
\emph{rationality constant} and it is the infimum over the
symplectic areas of all nonconstant spheres in $W$ with positive
area. More explicitly,
$$
\hbar :=\inf_{A\in \pi_2(W)} \big{\{}\left< \omega, A\right> \colon
 \left< \omega, A\right> >0\big{\}}.
$$

Denote by $c_1 :=c_1(W,J)\in H^2(W,\Z)$ the first Chern class of
$W$. The \emph{minimal Chern number} of a symplectic manifold $(W,
\omega)$ is the integer $N$ which generates the discrete group
$\left<c_1, \pi_2(W)\right> \subset \reals$ formed by the integrals
of $c_1$ over the spheres in $W$, i.e.,
$$
\left<c_1, \pi_2(W)\right> = N \Z
$$
where $N\in \Z^+.$ When $\left<c_1, \pi_2(W)\right>=0,$ we set
$N=\infty$. The constant $N$ is given explicitly by
$$N :=\inf_{A\in\pi_2(W)} \{ \left< c_1,A\right>   \colon \left< c_1,A\right> >0 \}.$$

A symplectic manifold $(W,\omega)$ is called \emph{monotone}
(\emph{negative monotone}) if the cohomology classes $c_1$ and
$[\omega]$ satisfy the condition
$$
{c_1}|_{\pi_2(W)} = \tau \;[\omega]|_{\pi_2(W)}
$$
for some non-negative (respectively, negative) constant $\tau\in
\reals.$

The manifold $(W,\omega)$ is called \emph{symplectically aspherical}
if
$$
{c_1|}_{\pi_2(W)}=0=[\omega]|_{\pi_2(W)}.
$$
Notice that a symplectically aspherical manifold is monotone and a monotone (or negative monotone) manifold is rational.\\

All the Hamiltonians~$H$ on $W$ considered in this paper are assumed
to be compactly supported and one-periodic in time, namely,
$$
H\colon S^1\times W \rightarrow \reals,
$$
where $S^1=\reals / \Z,$ and we set $H_t=H(t,\cdot)$ for $t\in S^1$.
The Hamiltonian vector field $X_H$ of $H$ is defined by
$\iota_{X_H}\omega=-dH.$ The time-one map of the flow of the
Hamiltonian vector field $X_H$ is called a \emph{Hamiltonian
diffeomorphism} and denoted by $\varphi_H.$

The composition $\varphi_H^t \circ \varphi_K^t$ of two Hamiltonian
flows is again Hamiltonian and it is generated by $K\#H$ where
\begin{eqnarray}
(K\#H)_t :=K_t + H_t \circ (\varphi_K^t)^{-1}.
\end{eqnarray}
In general, $K\#H$ is not a one-periodic Hamiltonian. However,
$K\#H$ is one-periodic if $H_0 = 0 = H_1$. This condition can be met
by reparametrizing the Hamiltonian as a function of time without
changing the time-one map. Thus, in what follows, we will usually
treat $K\#H$ as a one-periodic Hamiltonian.

The \emph{Hofer norm} of a one-periodic Hamiltonian~$H$ is defined
by
$$
||H||:=\displaystyle\int_0^1(\displaystyle\max_W H_t -
\displaystyle\min_W H_t) dt.
$$
The Hamiltonian diffeomorphism~$\varphi_H$ is said to
\emph{displace} a subset $U$ of $W$ if
$$
\varphi_H(U)\cap U =\emptyset.
$$
When such a map exists, we call $U$ \emph{displaceable} and define
the \emph{displacement energy of $U$} to be
$$
e(U):=\inf\{||H|| \colon  \varphi_H\;\mbox{displaces}\; U \}
$$
where $||\cdot||$ is the Hofer norm.

\subsection{Capped Periodic Orbits and Floer Homology}
\label{subsection:capped orbits} Let $x\colon S^1 \rightarrow W$ be
a contractible loop with capping $u\colon D^2 \rightarrow W,$ i.e.,
$u|_{\partial D^2}=x$. Two cappings $u$ and $v$ of $x$ are called
\emph{equivalent} if the integrals of $\omega$ and of $c_1$ over the
sphere obtained by attaching $u$ to $v$ are both equal to zero. For
instance, when $W$ is symplectically aspherical, all cappings of $x$
are equivalent. A \emph{capped closed curve} $\bx$ is, by
definition, a closed curve $x$ equipped with an equivalence class of
cappings.

\subsubsection{Hamiltonian Action and the Mean Index}
\label{subsubsection:ham_act__mean_ind} The \emph{action functional}
of a one-periodic Hamiltonian~$H$ on a capped closed curve
$\bx=(x,u)$ is defined by
$$
\A_H(\bx) :=-\int_u \omega + \int_{S^1} H_t(x(t)) dt.
$$
The space of capped closed curves is a covering space of the space
of contractible loops and the critical points of the action
functional are exactly the capped one-periodic orbits of the
Hamiltonian vector field $X_H.$ The \emph{action spectrum} $\CS(H)$
of $H$ is the set of critical values of the action.

A (capped) periodic orbit $\bx$ of $H$ is said to be
\emph{non-degenerate} if the linearized return map
$$
d\varphi_H\colon T_{x(0)}W \rightarrow T_{x(0)}W
$$
has no eigenvalues equal to one. Note that capping has no effect on
degeneracy or non-degeneracy of $\bx.$

Using a trivialization of $x^*TW$ arising from the capping of $\bx,$
the linearized flow along $x$
$$d\varphi_H^t\colon T_{x(0)}W \rightarrow T_{x(t)}W$$
can be viewed as a symplectic path $\Phi\colon[0,1] \rightarrow
\Sp(2n).$ The mean index of $\bx$ is defined by $\D(\bx)
:=\D(\Phi);$ see Definition~\ref{def:mean_index}. When we need to
emphasize the role of $H$, we write $\D_H(\bx).$ A list of
properties of the mean index can be found in
section~\ref{section:coiso_maslov_index}. In general, the mean index
and the action depend on the equivalence class of the capping $u$ of
the loop $x.$ More concretely, let $A$ be a $2$-sphere and denote by
$\bx\#A$ the recapping of $\bx$ by attaching $A.$ Then we have
$$
\D(\bx\#A)=\D(\bx)-2\left<c_1,A\right>\quad \mbox{and} \quad
\CA_H(\bx\#A)=\CA_H(\bx) -\int_A \omega.
$$

\subsubsection{Conley-Zehnder Index and Floer Homology}
Consider a non-degenerate path $\Phi\colon [0,1] \rightarrow
\Sp(2n),$ i.e., such that $\Phi(1)$ has no eigenvalues equal to one.
We denote by $\MUCZ(\Phi)$ the \emph{Conley-Zehnder index} of
$\Phi.$ For a non-degenerate capped closed orbit $\bx=(x,u),$ its
Conley-Zehnder index is given by the Conley-Zehnder index of the
symplectic path $\Phi$ obtained from the linearized flow
$d\varphi_H^t$ and a trivialization arising from the capping $u$. Up
to a sign, it is defined as in~\cite{Sa,SZ} and we use the
normalization such that $\MUCZ(\bx) = n$ when $\bx$ is a
non-degenerate maximum (with trivial capping) of an autonomous
Hamiltonian with small Hessian; cf.~\cite{GG:action}.

We have the following relation between the Conley-Zehnder and mean
indices for non-degenerate paths and orbits; cf.~\cite{SZ}:
\begin{eqnarray}
\label{eqnar:bounds_mean_CZ} |\D(\Phi)-\MUCZ (\Phi)|<n
\quad\mbox{and hence}\quad |\D(\bx)-\MUCZ(\bx)|<n.
\end{eqnarray}

Let us recall the definition of the Floer homology for a
non-degenerate Hamiltonian~$H$. The Floer chain groups are generated
by the capped one-periodic orbits of $H$ and graded by the
Conley-Zehnder index. The boundary operator is defined by counting
solutions of the Floer equation
$$
\frac{\p u}{\p s}+ J_t(u)\frac{\p u}{\p t}=-\bigtriangledown H_t(u)
$$
with finite energy. Floer trajectories for a non-degenerate
Hamiltonian~$H$ with finite energy converge to periodic orbits $\bx$
and $\by$ as $s\rightarrow\pm \infty$ and satisfy
$$
E(u)=\CA_H(\bx)-\CA_H(\by)=\displaystyle\int_{-\infty}^{\infty}\int_{S^1}
\Big{|}\Big{|}\frac{\p u}{\p s}\Big{|}\Big{|}^2 dt ds.
$$
The boundary operator counts Floer trajectories converging to
periodic orbits $y$ and $x$ as $s\rightarrow\pm \infty$ and
satisfying the condition [(capping of $\bx$)\#$u$] = [capping of
$\by$].

This construction extends by continuity from non-degenerate
Hamiltonians to all Hamiltonians; see~\cite{Sa,SZ} for more details.
\begin{Remark}
\label{rmk:novikov} The total Floer homology is independent of the
Hamiltonian and, up to a shift of the grading and the effect of
recapping, is isomorphic to the homology of $W$. More precisely, we
have
$$
HF_*(H) \cong H_{*+n}(W) \otimes \Lambda
$$
as graded $\Lambda$-modules; see, for instance,~\cite{GG:conley,MS}
and references therein for details on the definition of the Novikov
ring $\Lambda$. In particular, the fundamental class $[W]$ can be
viewed as an element of $HF_n(H)$.
\end{Remark}

\begin{Remark} To ensure that the Floer differential is defined, throughout this
paper we either assume $W$ to be weakly monotone (see,
e.g.~\cite{HS, MS, On, Sa}) or utilize the machinery of virtual
cycles (see, e.g.,~\cite{FO, FOOO, LT}). In our main result, one of
the possible conditions on $W$ is negative monotonicity. In this
case, $W^{2n}$ is weakly monotone if and only if $N \geq n - 2$,
where $N$ is the minimal Chern number.
\end{Remark}

\subsubsection{Filtered Floer Homology and Homotopy} Let us recall the definition of the
filtered Floer homology for a non-degenerate Hamiltonian $H.$ The
(total) chain Floer complex $CF_*(H)=\colon CF_*^{(-\infty,
\infty)}(H)$ admits a filtration by $\reals$. For each
$b\in(-\infty,\infty]$ outside $\CS(H),$ the chain complex
$CF_*^{(-\infty, b)}(H)$ is generated by the capped one-periodic
orbits of $H$ with action $\A_H$ less than~$b.$ For $-\infty\leq a <
b \leq \infty$ outside $\CS(H),$ set
$$
CF_*^{(a,b)}(H) :=CF_*^{(-\infty,b)}(H) / CF_*^{(-\infty,a)}(H).
$$
The boundary operator $\p\colon CF_*(H)\rightarrow CF_{*-1}(H)$
descends to $CF_*^{(a,b)}(H)$ and hence the \emph{filtered Floer
homology} $HF_*^{(a,b)}(H)$ is defined.

This construction also extends by continuity to all Hamiltonians.
For an arbitrary (one-periodic in time) Hamiltonian $H$ on $W,$ set
\begin{eqnarray}
\label{eqnar:def_FHdeg} HF_*^{(a,b)}(H) :=HF_*^{(a,b)}(\tilde{H})
\end{eqnarray}
where $\tilde{H}$ is a non-degenerate perturbation of $H$ and
$-\infty\leq a < b \leq \infty$ are outside $\CS(H)$.

When $a<b<c,$ we have $CF_*^{(b,c)}(H)=CF_*^{(a,c)}(H) /
CF_*^{(a,b)}(H)$ and thus obtain the long exact sequence
\begin{eqnarray}
\label{eqnar:longexactsequence} \ldots \rightarrow HF_*^{(a,b)}(H)
\rightarrow HF_*^{(a,c)}(H) \rightarrow HF_*^{(b,c)}(H) \rightarrow
HF_{*-1}^{(a,b)}(H) \rightarrow \ldots .\\ \nonumber
\end{eqnarray}

By definition, a \emph{homotopy} of Hamiltonians on $W$ is a family
of (one-periodic in time) Hamiltonians $H^s$ smoothly parameterized
by $s\in\reals$ and such that $H^s\equiv H^0$ when $s$ is near
$-\infty$ and $H^s\equiv H^1$ when $s$ is near $\infty$;
see~\cite{Gi:coiso} and references therein for the definitions,
properties and proofs.

Set
$$
E :=\displaystyle\int_{-\infty}^{\infty}\int_{S^1}
\displaystyle\max_W \partial_sH_t^s dt ds.
$$
For every $C\geq E,$ the homotopy induces a map of the filtered
Floer homology, which we denote by $\Psi_{H^0H^1},$  shifting the
action filtration by $C$:
\begin{eqnarray}
\label{eqnar:morphism} \Psi_{H^0H^1}\colon HF_*^{(a,b)}(H^0)
\rightarrow  HF_*^{(a+C,b+C)}(H^1).
\end{eqnarray}

\begin{Example}
\label{example:linear_homotopy} Let $H^s$ be an \emph{increasing
linear homotopy} from $H^0$ and $H^1,$ i.e.,
$$
H^s=(1-f(s))H^0 + f(s)H^1
$$
where $f\colon\reals \rightarrow [0,1]$ is a monotone increasing
compactly supported function equal to zero near $-\infty$ and equal
to one near $\infty.$ Since
\begin{eqnarray}
\label{eqnar:linearhomotopy} E\leq \int_{S^1} \displaystyle\max_W
(H^1-H^0) dt,
\end{eqnarray}
we have the homomorphism $\Psi_{H^0H^1}$ for every $C\geq\int_{S^1}
\displaystyle\max_W (H^1-H^0) dt.$
\end{Example}

Furthermore, we have the following continuity property for filtered
homology: let $(a^s,b^s)$ be a family (smooth in $s$) of non-empty
intervals such that $a^s$ and $b^s$ are outside $\CS(H^s)$ for some
homotopy $H^s$ and such that $(a^s,b^s)$ is equal to $(a^0,b^0)$
when $s$ is near $-\infty$ and equal to $(a^1,b^1)$ when $s$ is near
$\infty.$ Then there exists an isomorphism of homology
\begin{eqnarray}
\label{eqnar:homotopy_homology} HF^{(a_0,b_0)}(H^0)
\stackrel{\cong}{\longrightarrow} HF^{(a_1,b_1)}(H^1).
\end{eqnarray}
When the interval is fixed and the homotopy is monotone decreasing,
the isomorphism~(\ref{eqnar:homotopy_homology}) is in fact
$\Psi_{H^0H^1}$ which in general is not the case.

\subsection{Stable Coisotropic Submanifolds and Maslov Index}
\label{subsection:stable_coiso} In this section, we give the
definition and some properties of stable coisotropic submanifolds.
This class of coisotropic submanifolds was introduced
in~\cite{Bo:fr, Bo:en} and is defined as follows.

The submanifold $M$ is said to be \emph{stable} if there exist $k$
one-forms $\alpha_1,\ldots,\alpha_k$ on $M$ such that
$$
\mbox{Ker}\;d\alpha_i \supset \mbox{Ker}\; \omega_M \quad \mbox{for
all } i=1,\ldots,k
$$
and
$$
\alpha_1\wedge\ldots\wedge\alpha_k\wedge\omega_M^{n-k} \not= 0 \quad
\mbox{on M}.
$$
Notice that this condition is rather restrictive. For instance, a
stable Lagrangian submanifold is necessarily a torus and a stable
coisotropic submanifold is automatically orientable. Thus, examples
of stable coisotropic submanifolds include Lagrangian tori and also
contact hypersurfaces. Moreover, the stability condition is closed
under products. For more details, we refer the reader
to~\cite{Bo:fr, Bo:en, Gi:coiso}.

As a consequence of the Weinstein symplectic neighborhood theorem,
we obtain tubular neighborhoods of stable coisotropic submanifolds:
\begin{Proposition}[\cite{Bo:fr, Bo:en}]
\label{prop:nbhd} Let $M^{2n-k}$ be a closed stable coisotropic
submanifold of $(W^{2n}, \omega)$. Then, for $r>0$ sufficiently
small, there exists a neighborhood of $M$ in $W$ which is
symplectomorphic to
$$
U_r = \{ (q,p) \in M \times \reals^k \colon  |p|< r\}
$$
equipped with the symplectic form
$$
\omega=\omega_{M} + \displaystyle\sum_{j=1}^k d(p_j\alpha_j)
$$
where $p=(p_1,\ldots,p_k)$ are the coordinates in $\reals^k$ and
$|p|$ is the Euclidean norm of~$p$.
\end{Proposition}

Thus, such a neighborhood is foliated by a family of coisotropic
submanifolds $M_p = M \times \{p\}$ with $p\in B^k_r \colon =\{
p\in\reals^k\colon |p|< r\}$ and a leaf of the characteristic
foliation on $M_p$ projects onto a leaf of the characteristic
foliation on $M.$

Furthermore, we have
\begin{Proposition}[\cite{Bo:fr, Bo:en, Gi:coiso}]
\label{prop:metric} Let $M^{2n-k}$ be a stable coisotropic
submanifold of $(W^{2n}, \omega)$. Then
\begin{itemize}
\item the leaf-wise metric $(\alpha_1)^2 + \ldots + (\alpha_k)^2$ on $\Ff$ is leaf-wise flat;
\item the Hamiltonian flow of $\rho=(p_1^2+\ldots+p_k^2)/2$ is the leaf-wise geodesic flow of this metric.
\end{itemize}
\end{Proposition}

Consider $\bx=(x,u)$ a non-trivial (capped) periodic orbit of the
Hamiltonian flow of $\rho.$ Then, as a consequence of
Proposition~\ref{prop:Maslov_welldefined}, we obtain that the mean
index $\D_{\rho}(\bx)$ of a periodic orbit $\bx$ of a leaf-wise
geodesic flow on $M$ is equal to, up to a sign, the coisotropic
Maslov index of the projection of $\bx$ on $M.$ More precisely,
\begin{eqnarray}
\label{eqnar:proj} \mu(\pi(x), \hu) = - \D_{\rho} (x,u)
\end{eqnarray}
where $\hu$ is the capping of the orbit $\pi(x)$ given by the
capping $u$ of $x$ together with the cylinder obtained from the
projection of $x$ on $M;$ see Figure~\ref{fig:u_hat}.
\begin{figure}[!htb]
  \centering
  \def\svgwidth{190pt}
  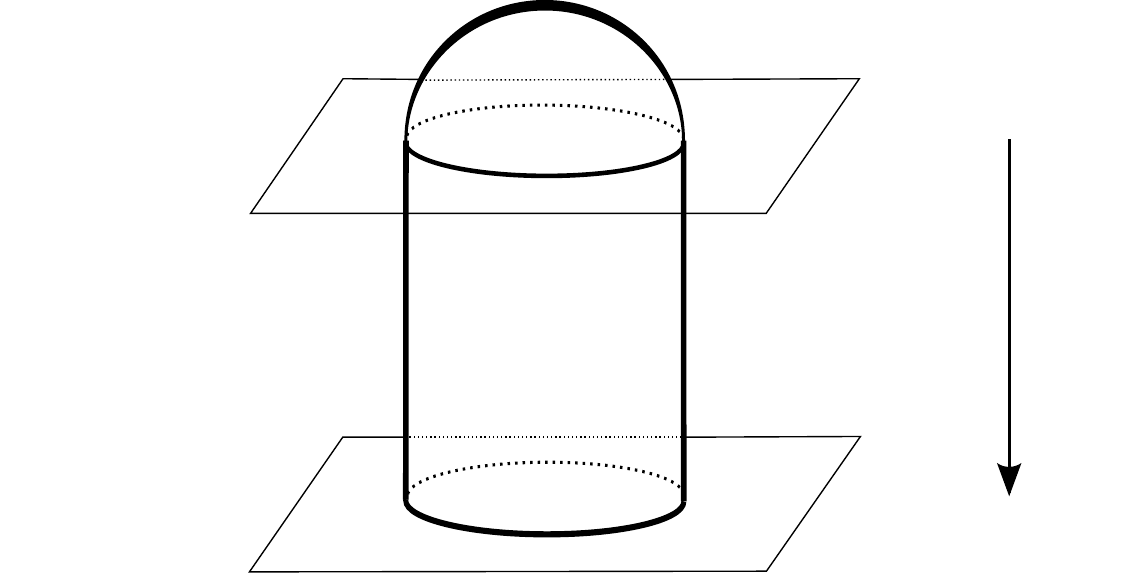
  \caption{Capping $\hu$.}\label{fig:u_hat}
\end{figure}

The following result establishes bounds on the Conley-Zehnder index
of a small non-degenerate perturbation of a capped periodic orbit
$(x,u)$ of $\rho$ which goes beyond~(\ref{eqnar:bounds_mean_CZ}).
(Here as above $M$ is stable.)
\begin{Proposition}[\cite{Gi:maslov}]
\label{prop:bound} Let $\rho'$ be a small perturbation of the
Hamiltonian $\rho$ defined in Proposition~\ref{prop:metric} and $x'$
a non-degenerate periodic orbit of $\rho'$ (with a capping $u'$)
close to a non-trivial periodic orbit $x$ of $\rho$ (with a capping
$u$). Then
$$
\D_{\rho}(x,u) - n \leq \MUCZ((x,u)') \leq \D_{\rho}(x,u) + (n-k)
$$
where $(x,u)':=(x',u').$
\end{Proposition}
\section{Proof of the Main Theorem}

\subsection{``Pinned" Action Selector} One of the tools used
in the proof of Theorem~\ref{theorem:main} is an action selector
defined for ``pinned" Hamiltonians. This tool was first introduced
in~\cite{Ke:pinned, KS} for a class of Hamiltonians and manifolds
which are somewhat different from those we work with. However, the
definition of the action selector is essentially the same. In this
section, we describe this action selector and a \emph{special orbit}
associated with it.

Let $W$ be a rational symplectic manifold and $U$ an open
neighborhood of the coisotropic submanifold $M$ of $W.$ Consider
$K\colon W \rightarrow \reals$ a compactly supported autonomous
Hamiltonian such that the neighborhood $U$ contains the support of
$K$, $\supp K,$ and $U$ is displaced by a Hamiltonian $H.$ We may
assume $H$ is non-negative with minimum value equal to zero. Suppose
that $K$ is constant on $M$ where it attains its maximum value $\max
K=:\lambda,$ the maximum value $\lambda$ is greater than $||H||$ and
that $K$ is strictly decreasing and $C^2$-close to $\lambda$ on a
small neighborhood of $M.$

Consider the quotient map $j_K\colon HF_n(K) \rightarrow
HF_n^{(\lambda-\delta,\lambda+\delta)}(K)$ and define the element
$[\max_K]\in HF_n^{(\lambda-\delta,\lambda+\delta)}(K)$ as
$$
[\mbox{max}_K]:= j_K([W])
$$
where the fundamental class $[W]$ is seen as an element of
$HF_n(K)$; recall Remark~\ref{rmk:novikov}.
\begin{Definition}[\emph{``Pinned" Action Selector}]\label{def:actionselector}For $\delta>0$ small and
$\alpha>\lambda +\delta,$ consider the inclusion map
$$
\iota_{\alpha}\colon HF_n^{(\lambda-\delta, \,\lambda+\delta)}(K)
\hook HF_n^{(\lambda-\delta, \,\alpha)}(K).
$$
Define
$$
c(K) :=\displaystyle\inf_{\delta >0} \inf\{ \alpha> \lambda + \delta
 \colon  \iota_{\alpha}([\mbox{max}_K])=0\}.
$$
\end{Definition}
We have $c(K)\in\CS(K)$ and $c(K)=\CA_{K}(\bx)$ for some capped
orbit $\bx$ which is called a \emph{special one-periodic orbit}.
\begin{Claim}
\label{claim:maxK_exact} There exists $\mathcal{N}\in
HF_{n+1}^{(\lambda + \delta, \,\infty)}(K)$ such that $\p
\mathcal{N}=[\max_{K}]$ where
$$
\p \colon HF^{(\lpd,\,\lpd + ||H||)}_{n+1}(K) \rightarrow
HF^{(\lmd,\,\lpd)}_n(K)
$$
is the \it{connecting differential} in the long exact
sequence~(\ref{eqnar:longexactsequence}) (with $a=\lmd,\;b=\lpd$ and
$c=\lpd + ||H||$).
\end{Claim}
\begin{proof}
For $\delta>0$ sufficiently small, namely such that $\lambda -\delta
> ||H||,$ consider the following commutative diagram:
\begin{equation*}
\xymatrix{
{}                                                      & {\HF^{(\lpd,\,\lpd + ||H||)}_{n+1}(K)} \ar[d]^{\p} \\
{}                                                      & {\HF^{(\lmd,\,\lpd)}_n(K)} \ar[d]^{\iota}\\
{\HF^{(\lmd - ||H||,\,\lpd)}_n(K)}  \ar[r]^{\Psi\circ\Phi} \ar[d]_{\Phi} & {\HF^{(\lmd ,\,\lpd + ||H||)}_n(K)}\\
{\HF^{(\lmd,\,\lpd + ||H||)}_n(K \# H)} \ar[ur]^{\Psi}
\ar[r]^{\Theta}    & {\HF^{(\lmd,\,\lpd + ||H||)}_n(H)} }
\end{equation*}
where $\iota$ is the inclusion and $\p$ is the \emph{connecting
differential} in the long exact
sequence~(\ref{eqnar:longexactsequence}) (with $a=\lmd,\;b=\lpd$ and
$c=\lpd + ||H||$). The maps $\Phi$ and $\Psi$ are induced by
monotone homotopies between $K$ and $K\# H$: the map $\Phi$ is
induced by the linear monotone increasing homotopy from $K$ to
$K\#H$ (recall that $H\geq 0$) where, in
Example~\ref{example:linear_homotopy}, $C=||H||;$ the map $\Psi$ is
induced by the linear monotone decreasing homotopy from $K\#H$ to
$K$ where, in~(\ref{eqnar:morphism}), $C=0.$

Since $\varphi_H$ displaces $\supp K,$ the one-periodic orbits of
$K\#H$ are exactly the one-periodic orbits of $H$ and moreover
$\CS(K\#H)=\CS(H);$ see~\cite{HZ}. Then the map $\Theta$ is an
isomorphism induced by a linear monotone homotopy between $K\#H$ and
$H$ due to the continuity property~(\ref{eqnar:homotopy_homology})
of filtered homology.

Note that the vertical part of the diagram, which consists of the
maps $\p$ and $\iota,$ is part of a long exact sequence as
in~(\ref{eqnar:longexactsequence}).

Consider the projection
$$
j_H\colon HF(H) \rightarrow HF^{(\lmd,\lpd + ||H||)}(H).
$$
and the image
$$
j_H([W])\in HF^{(\lmd,\lpd+ ||H||)}(H)
$$
of the class $[W]\in HF_n(H).$ Since
$$
\lambda - \delta >||H||,
$$
we have
$$
0= j_H([W]) \in \HF^{(\lmd,\,\lpd + ||H||)}_n(H)
$$
and hence
$$
\HF^{(\lmd,\,\lpd + ||H||)}_n(K) \ni\Psi\circ\Theta^{-1}\circ
j_H([W])=\iota([\mbox{max}_K])=0
$$
where the first equality follows from the fact that $j_H([W])$ is
equal to the image $\Theta \circ\Phi\circ j([W])$ of the class $[W]$
seen as an element of $HF_n(K)$ and the map $j$ is the projection
$$
j\colon HF_n(K) \rightarrow HF_n^{(\lmd-||H||,\lpd)}(K).
$$
Then
$$
0=[\mbox{max}_K]\in \HF^{(\lmd ,\,\lpd)}_n(K)
$$
and, since $\iota$ and $\p$ are part of a long exact sequence, it
follows that there exists $\mathcal{N} \in \HF^{(\lpd, \lpd +
||H||)}_{n+1}(K)$ such that
$$
\p \mathcal{N}=[\mbox{max}_{K}]\in \HF^{(\lmd ,\,\lpd)}_n(K).
$$
\end{proof}

Consider a small non-degenerate perturbation $K' \colon S^1\times W
\rightarrow W$ of $K$ with $\max K'=\lambda$ and such that
\begin{eqnarray}
\label{eqnar:pert_floerhomol}
HF_j^{(a_0,a_1)}(K):=HF_j^{(a_0,a_1)}(K')
\end{eqnarray}
with $a_0,a_1\not\in\CS(K),\CS(K')$; recall
definition~(\ref{eqnar:def_FHdeg}).

Consider the class $[\max_{K'}]:= j_{K'}([W]) \in
HF_n^{(\lambda-\delta,\lambda+\delta)}(K')$ and define
$$
c(K'):=\displaystyle\inf_{\delta >0} \inf\{ \alpha> \lambda + \delta
 \colon \iota_{\alpha}([\mbox{max}_{K'}])=0\}.
$$
where $\iota_{\alpha}\colon HF_n^{(\lambda-\delta,
\,\lambda+\delta)}(K') \hook HF_n^{(\lambda-\delta, \,\alpha)}(K')$
is the inclusion map. We have $c(K') \rightarrow c(K)$ as $K'
\rightarrow K$ and $c(K')=\CA_{K'}(\bx')$ for some capped orbit
$\bx'.$ A special one-periodic orbit $\bx'$ for $K'$ is obtained
explicitly the following way: by~(\ref{eqnar:pert_floerhomol}) and
Claim~\ref{claim:maxK_exact}, we obtain a class $[\bc']\in
HF_{n+1}^{(\lambda + \delta, \,\infty)}(K')$ such that
$\p[\bc']=[\max_{K'}]$. Within each chain $\bc'$ pick a capped orbit
with the largest action and then among the resulting capped orbits
choose a capped orbit $\bx'$ with the least action. Moreover, we
have $\MUCZ(\bx')=n+1.$

\begin{Remark}
\label{rmk:conn_orb} The orbit $\bx'$ does not have to be connected
with the constant orbit $(\gamma_p,u_p)$ by a Floer downward
trajectory. However, there exists a capped orbit $\by'$ with this
property and such that
$$
 \lambda \leq \CA_K(\by') \leq \CA_K(\bx').
$$
The orbit $\by'$ is given explicitly by the following construction:
take all chains $\bc'$ such that $\p[\bc']=[\max_K'].$ Within each
chain consider a capped orbit connected to $(\gamma_p,u_p)$ with the
least action and among these orbits consider one with the least
action, $\by'.$
\end{Remark}

For a Hamiltonian $K$ as above, consider a sequence $(K_j)$ such
that $K_j$ is as $K'$ above and $K_j\rightarrow K$ as
$j\rightarrow\infty.$ By the Arzela-Ascoli theorem, there exists a
subsequence of special one-periodic orbits $\bx_j$ which converges
to an orbit $\bx$ of $K$ which is called a \emph{special
one-periodic orbit} of $K.$ Recall that $c(K_j) \rightarrow c(K)$ as
$j\rightarrow\infty$ and $\MUCZ(\bx_j)=n+1.$

The following results give upper and lower bounds for the action of
a special one-periodic orbit.
\begin{Lemma}
\label{lem:actionbounds} For a special one-periodic orbit $\bx$ of
$K,$ we have the following action upper bound:
\begin{eqnarray}
\label{eqnar:action_upperbound} \CA_K(\bx) \leq  \lambda + ||H||.
\end{eqnarray}
\end{Lemma}
\begin{proof}
Since $\iota ([\max_K])=0$ (proved in Claim~\ref{claim:maxK_exact}),
$c(K)\leq \lambda + ||H||.$ By the definition of the ``pinned"
action selector, we have $c(K)\geq \lambda.$ Then the result follows
immediately from the fact that $\bx$ is a carrier of the action
selector $c.$
\end{proof}

\begin{Lemma}
\label{lem:actionbounds2} A capped loop $\bx$ as in Lemma
\ref{lem:actionbounds} satisfies
\begin{eqnarray}
\label{eqnar:action_lowerbound} \CA_K(\bx) - \lambda \geq \epsilon
\end{eqnarray}
where $\epsilon > 0$ is independent of $K.$
\end{Lemma}
\begin{proof}
Consider a sequence $(K_j)$ as above. Let $u_j$ be a Floer downward
trajectory connecting the orbit $\by_j$ defined in
Remark~\ref{rmk:conn_orb} and the constant orbit $(\gamma_p,u_p).$
If $E(u_j)$ is below $\hbar$, then we may apply a similar argument
to that in lemmas~$6.2$ and~$6.4$ in~\cite{Gi:coiso} which draws
heavily from~\cite{{Bo:fr},{Bo:en}} and we obtain
$$
d< E(u_j) = \CA_{K_j}(\by_j) - \CA_{K_j} (\bg_p)
$$
where $d>0$ is independent of $K_j$. Define
$$
\epsilon := \max \{ \hbar, d\} >0.
$$
Then $E(u_j)=\CA_{K_j}(\by_j) - \CA_{K_j}(\bg_p) \geq \epsilon$ and,
since $\CA_{K_j}(\by_j) \leq \CA_{K_j}(\bx_j),$ it follows that
\begin{eqnarray}
\label{eqnar:action_lower_bound_pert} \CA_{K_j}(\bx_j) - \lambda
\geq \epsilon.
\end{eqnarray}
Then take~(\ref{eqnar:action_lower_bound_pert}) to the limit when $j
\rightarrow \infty$ and we obtain the desired result
$$
\CA_K(\bx) - \lambda \geq \epsilon.
$$
\end{proof}

\subsection{Proof of Theorem~\ref{theorem:main}}\label{subsection:proof_thm}
Fix $R$ such that $U_R=M\times B^k_R$ is defined by
Proposition~\ref{prop:nbhd}. Consider $\varepsilon>0$ small and
$0<r<R/2$. Assume $U_r$ is displaced by some Hamiltonian $H$ and
consider $\lambda> e(U_r).$ Let $\Klre\colon [0,R] \rightarrow
\reals$ be a smooth decreasing map such that
\begin{itemize}
\item $\Klre \geq 0$
\item $\Klre(0)= \lambda $
\item $\Klre $ is strictly decreasing and $C^2$-close to $\lambda$ on $[0, \varepsilon]$
\item $\Klre$ is concave on $[\varepsilon, 2\varepsilon]$
\item $\Klre$ is linear decreasing from $\lambda - \varepsilon $ to $\varepsilon$ on $[2\varepsilon, r-\varepsilon]$
\item $\Klre$ is convex on $[r-\varepsilon, r ]$
\item $\Klre \equiv 0$ on $[r,R].$
\end{itemize}
We also denote by $\Klre$ the Hamiltonian
$$
\Klre\colon W \rightarrow \reals
$$
defined by $\Klre(|p|)$ on $U_R$ and equal to zero outside $U_R$.

Fix $r$ and consider the family of functions $\Kle$ depending
smoothly on the parameters $\lambda$ and $\varepsilon$. These
Hamiltonians have the same properties as the Hamiltonian $K$ in the
previous subsection.

The key to the proof, as in~\cite{Gi:maslov}, is the following
result which gives the location of a sequence of special
one-periodic orbits $\bx_i.$
\begin{Lemma}[\cite{Gi:maslov}]
\label{prop:localize} There exists $\lambda > e(U_R)$ and a sequence
$\varepsilon_i \rightarrow 0$ such that a special one-periodic orbit
of $\Kle_i$ $\bx_i$ satisfies
$$
|p(x_i)|\in [\varepsilon_i, 2\varepsilon_i]
$$
where $ p=(p_1,\ldots,p_k)$ are the coordinates introduced in
Proposition~\ref{prop:nbhd}.
\end{Lemma}
\begin{Remark}
In~\cite{Gi:maslov}, the result of Lemma~\ref{prop:localize} is
proved for a class of Hamiltonians which is slightly different from
the one we work with. However the above lemma holds for the same
reasons as the result in the referred paper.
\end{Remark}
By Proposition~\ref{prop:metric}, if we reparametrize $\bx_i$ and
reverse its orientation, then $\bx_i$ can be viewed as a periodic
orbit $\bx_i^-$ of $\rho.$ Since the slopes of the Hamiltonians
$K_{\lambda, \varepsilon_i}$ are bounded from above (for instance,
by $2\lambda/ r$), then (by the Arzela-Ascoli theorem) we define
$$
\bg \colon= \mbox{limit of (a subsequence of)} \;(\pi(x_i^-),
\widehat{u_i}^-).
$$
where $\mu(\pi(x_i^-), \widehat{u_i}^-)= - \D_{\rho}(x_i^-,u_i^-)$
by~(\ref{eqnar:proj}). Then, by~(\ref{eqnar:bounds_mean_CZ}),
$$
-n\leq \MUCZ((x_i^-,u_i^-)') - \D(x_i^-,u_i^-)\leq n
$$
and hence
\begin{eqnarray*}
 - n \leq \mu(\pi(x_i^-), u_i^-)\quad + &\MUCZ((x_i^-,u_i^-)')& \leq n \\
&\scriptsize{\|}&\\
& - \MUCZ ((x_i,u_i)')& = -(n+1)
\end{eqnarray*}
where the first equality uses the fact that $x_i$ is in the region
where $\Kle_i$ is concave, i.e., where
$|p(x_i)|\in[\varepsilon_i,2\varepsilon_i]$ and we obtain the
following bounds for the Maslov index of
$(\pi(x_i^-),\widehat{u_i}^-)$:
\begin{eqnarray}
\label{eqnar:bounds_maslov} 1\leq \mu(\pi(x_i^-), u_i^-) \leq 2n+1.
\end{eqnarray}
Considering the limit (of a subsequence)
of~(\ref{eqnar:bounds_maslov}), we have
\begin{eqnarray}
\label{eqnar:bounds_maslov_limit} 1\leq \mu(\bg) \leq 2n+1.
\end{eqnarray}

By Proposition~\ref{prop:nbhd}, we obtain
\begin{eqnarray}
\label{eqnar:area}
A_{K_{\lambda,\varepsilon_i}}(\bx_i) &=& K_{\lambda,\varepsilon_i}(\bx_i) -\int_{u_i} \omega\nonumber\\
&=& K_{\lambda,\varepsilon_i}(\bx_i) -\int_{\hat{u}_i} \omega -
|p(x_i)|l(\pi(x_i))
\end{eqnarray}
where $\hu_i$ is constructed as in
section~\ref{subsection:stable_coiso}; see Figure~\ref{fig:u_hat}.

Moreover,
by~(\ref{eqnar:action_upperbound}),~(\ref{eqnar:action_lowerbound})
and~(\ref{eqnar:area}), we have
\begin{eqnarray}
\label{eqnar:beforelim} 0 < \eps \leq
K_{\lambda,\varepsilon_i}(x_i^-) -\int_{{\hat{u´}_i}^-} \omega -
|p(x_i^-)|l(\pi(x_i^-)) - \lambda \leq e(U_r).
\end{eqnarray}
Since $|p(x_i^-)|\in [\varepsilon_i,2\varepsilon_i],\;
\Kle_i(x_i^-)\in [\varepsilon_i, \lambda-\varepsilon_i]$ and the
sequence $l(\pi(x_i^-))$ is bounded (since the slope of $\Kle_i$ is
bounded), then, taking the limit (of a subsequence)
of~(\ref{eqnar:beforelim}), we obtain
\begin{eqnarray}
\label{eqnar:bounds_area} 0 < \eps\leq \mbox{Area}(\bg) \leq e(U_r).
\end{eqnarray}
Recall that $\eps$ is independent of $\varepsilon_i.$ Then, taking
$r>0$ sufficiently small, we have
$$
0< \mbox{Area}(\bg) \leq e(M) + \varepsilon.
$$

Hence, we have the desired bounds for the area of $\bg.$ To obtain
the Maslov index bounds as presented in the theorem (which go
beyond~(\ref{eqnar:bounds_maslov_limit})), we will first prove that
the orbit $\gamma$ is non-trivial. Assume the contrary, that is,
that $\gamma$ is a trivial orbit. Then,
by~(\ref{eqnar:bounds_area}), the capping $v$ of $\gamma$ must be
non-trivial. Recall that we have one of the following conditions:
        \begin{itemize}
        \item $W$ is negative monotone,
        \item $e(M)< \hbar,$
        \item $2n+1<2N$.
        \end{itemize}
Suppose that $W$ is negative monotone. Then, $\left< c_1,v \right>$
and Area($\bg$) have opposite signs. However,
by~(\ref{eqnar:bounds_maslov_limit}) and~(\ref{eqnar:bounds_area}),
they are both positive and we obtain a contradiction. If $e(M)<
\hbar$ or $2n+1<2N$, we obtain contradictions by the definition of
the rationality constant $\hbar$ and~(\ref{eqnar:bounds_area}) or by
the definition of the minimal Chern number $N$
and~(\ref{eqnar:bounds_maslov_limit}), respectively. Therefore,
$\gamma$ is a non-trivial orbit. Furthermore, there exists a
(sub)sequence of non-trivial orbits $x_i$ as in
Lemma~\ref{prop:localize} which converges to $\gamma$. Then, by
Proposition~\ref{prop:bound}, we have
\begin{eqnarray*}
-\mu(\pi(x_i^-), u_i^-) - n \leq &\MUCZ((x_i^-,u_i^-)')& \leq - \mu (\pi(x_i^-), \widehat{u_i}^-) + n-k\\
&\scriptsize{\|}&\\
& - \MUCZ ((x_i,u_i)')& = -(n+1)
\end{eqnarray*}
where the first equality uses the fact that $x_i$ is in the region
where $\Kle_i$ is concave, i.e., where
$|p(x_i)|\in[\varepsilon_i,2\varepsilon_i]$. Then
$$
1\leq \mu(\pi(x_i^-), \widehat{u_i}^-) \leq 2n+1-k
$$
and considering the limit (of a subsequence) we obtain the desired
bounds for the Maslov index of $\bg$:
$$
1\leq \mu(\bg) \leq 2n+1-k.
$$

\section{Nearby Existence Theorem} \label{section:nearby} The theorem given in this
section guarantees the existence of periodic orbits for a dense set
of energy levels in a certain class of rational symplectic
manifolds. This result is proved for symplectically aspherical
manifolds in~\cite{Gi:coiso}. The structure of our proof is
essentially the same as in the referred paper and the necessary
changes are contained in the proof of Theorem~\ref{theorem:main}.

Let $W$ be a closed rational symplectic manifold and consider a map
$\overrightarrow{F}=(F_1,\ldots,F_k)\colon W\rightarrow \reals^k$
whose components $F_j$ are Poisson-commuting Hamiltonians, i.e.,
$\{F_i,F_j\}=0$ for $i\not= j$ and satisfy $dF_1\wedge\ldots\wedge
dF_k \not= 0 $ in $M_0$ where $M_a:=\overrightarrow{F}^{-1}(\{a\}),$
for $a\in \reals^k,$ and $M_0$ is a displaceable coisotropic
submanifold of $W$ with codimension $k.$ Assume that one of the
following conditions is satisfied
        \begin{itemize}
        \item $W$ is negative monotone,
        \item $e(M_0)< \hbar$,
        \item $2n+1<2N$.
        \end{itemize}
Then we have the following nearby existence result.
\begin{Theorem}
For a dense set of regular values $a\in\reals^k$ near the origin,
the level set $M_a$ carries a closed curve $x$ (with capping $u$ in
$W$) tangent to the characteristic foliation $\Ff_a$ on $M_a$.
\end{Theorem}
\begin{proof}
We prove the existence of an orbit (with the required properties) in
a level $M_a$ arbitrarily close to $M_0$ and the wanted result
follows immediately. Consider $K:=f(F_1,\ldots,F_k)$ where
$f\colon\reals^k \rightarrow \reals$ is a bump function supported in
a small neighborhood of the origin in $\reals^k$ and such that the
maximum value of $f$ is large enough. Since the support of $f$ is
small, we may assume that the support of $K$ is displaceable and all
$a\in \supp f$ are regular values of $\overrightarrow{F}.$ Hence the
coisotropic manifolds $M_a$ are compact and close to $M_0$ when
$a\in\reals^k$ is near the origin. By lemmas~\ref{lem:actionbounds}
and~\ref{lem:actionbounds2}, there exists a capped one-periodic
orbit of $K$ (in some regular level $M_a$) such that
\begin{eqnarray}
\label{eqnar:dense_action_bounds} \max K < A_K(\bx)\leq \max K +
||H||
\end{eqnarray}
where $H$ displaces $\supp K.$ The capped orbit $\bx$ can be
approximated by non-degenerate capped orbits with Conley-Zehnder
index equal to $n+1$ and hence, by~(\ref{eqnar:bounds_mean_CZ}), we
obtain
$$
1\leq \D(\bx)\leq 2n+1.
$$
Since one of the three conditions mentioned above is satisfied, the
orbit $x$ is non-trivial. Indeed, assume that $x$ is a trivial
orbit. Then~(\ref{eqnar:dense_action_bounds}) is equivalent to
$$
0<\mbox{Area}(\bx)\leq e(M).
$$
Then using the area and (mean) index bounds on $\bx$ and assuming
one of the above three conditions, we obtain a contradiction
(following the same reasoning as in
section~\ref{subsection:proof_thm}).

Furthermore, since the Hamiltonian $K$ Poisson-commutes with all
$F_j,$ the orbit $x$ is tangent to the characteristic foliation
$\Ff_a$ on $M_a$.
\end{proof}

\section{Appendix: The Coisotropic Maslov Index is Well Defined}
\label{section:coiso_maslov_index}

The objective of this section is to revisit the definition of the
coisotropic Maslov index and give a direct proof of the fact that it
is well defined. As mentioned in the introduction, similar notions
of index are originally considered in~\cite{Gi:maslov,Zi}.

First, we define the Maslov index of a loop of coisotropic subspaces
of $(\reals^{2n},\omega_0)$ where $\omega_0 := dx\wedge dy$ and
$(x,y)$ are the coordinates in $\reals^{2n}=\reals^n\times\reals^n.$
Then, we define the Maslov index of a capped loop lying in a
coisotropic submanifold and tangent to the characteristic foliation
of the coisotropic submanifold. We start by recalling the definition
of the mean index given in~\cite{SZ}. For its construction, we need
a collection of mappings given by the following theorem:
\begin{Theorem}[\cite{SZ}]
There is a unique collection of continuous mappings
$$
  \rho\colon\Sp(V,\omega) \rightarrow S^1
$$
(one for every symplectic vector space $V$) satisfying the following
conditions:
    \begin{itemize}
    \item {Naturality:} If $T\colon (V_1, \omega_1) \rightarrow (V_2, \omega_2)$ is a symplectic isomorphism (that is, $T^*\omega_2=\omega_1$), then
    $$\quad \rho(T\varphi T^{-1})=\rho(\varphi)$$
    for $\varphi \in \Sp(V_1, \omega_1)$.
    \item {Product:} If $(V, \omega)=(V_1 \times V_2, \omega_1\times \omega_2),$ then
    $$ \rho(\varphi)=\rho(\varphi_1)\rho(\varphi_2)$$
    for $\varphi\in \Sp(V,\omega)$ of the form $\varphi(z_1,z_2)=(\varphi_1z_1,\varphi_2z_2)$ where $\varphi_i\in \Sp(V_i, \omega_i)$.
    \item{Determinant:} If $\;\varphi\in \Sp(2n)\cap O(2n) \simeq U(n)$ is of the form
    $$
    \varphi= \left( \begin{array}{cc}
                                X & -Y \\
                                Y & X\end{array} \right),
    $$
    then
    $$\rho(\varphi)=\det(X+iY)$$
    \item{Normalization:} If $\varphi$ has no eigenvalues on the unit circle, then
    $$\quad\rho(\varphi)= \pm 1$$
    \end{itemize}
\end{Theorem}
\begin{Remark} \label{rmk:def_rho} The map $\rho\colon \Sp(2n) \rightarrow S^1$ is given explicitly by
$$
\rho(\varphi) := (-1)^{m_0} \displaystyle\prod_{\lambda \in
\sigma(\varphi)\cap S^1\backslash\{-1,1\}} \lambda^{m_+{(\lambda)}}
$$
where $\sigma(\varphi)$ is the set of eigenvalues of $\varphi,\;m_0$
is given by
$$
m_0 := \# \{\{\lambda, \lambda^{-1}\}\colon \lambda \in
\sigma(\varphi)\cap \mathbb{R}^- \}
$$
and $m_+(\lambda)$ is some multiplicity assigned to an eigenvalue
$\lambda\in S^1\backslash \{-1,1\};$ see page $1316$ in~\cite{SZ}
for the details of the definition of $m_+$.

Notice that only the eigenvalues of $\varphi$ on the unit circle and
on the negative real axis contribute to $\rho(\varphi)$.
\end{Remark} Then, the definition of the mean index of a path
$\Psi\colon[0,1]\rightarrow \Sp(2n)$ is given by:
\begin{Definition}[\emph{Mean Index}; \cite{SZ}]
\label{def:mean_index} Let $\Psi\colon[0,1]\rightarrow \Sp(2n)$ be a
path of symplectic matrices. Then choose a function
$\alpha\colon[0,1] \rightarrow \reals$ such that $\rho(\Psi_t) =
e^{\pi i\alpha(t)}.$ The \emph{Mean index} of the path $\Psi$ is
defined by
$$
\D (\Psi) := \alpha(1)-\alpha(0)
$$
\end{Definition}
The mean index $\D$ has the following properties:
\begin{enumerate}
\item Homotopy Invariance: $\D(\Psi)$ is an invariant of homotopy of $\Psi$ with fixed end points
\item \label{propertiesMeanIndex_concatenation} Concatenation: $\D$ is additive with respect to
concatenation of paths: $$\D(\Psi)=\D(\Psi|_{[0,a]}) +
\D(\Psi|_{[a,1]})$$ where $0<a<1$
\item \label{propertiesMeanIndex_loops} Loop: $\D(\varphi\Psi)=\D(\varphi)+\D(\varphi_0\Psi)$
if either $\varphi$ or $\Psi$ is a loop
\item \label{propertiesMeanIndex_conjinv} Naturality: $\D(T\Psi T^{-1})= \D(\Psi)$
where $T\colon (V_1,\omega_1)\rightarrow (V_2,\omega_2)$ is a
symplectic isomorphism and $\Psi\in\Sp(V_1, \omega_1)$
\item \label{propertiesMeanIndex_product} Product: $\D(\Psi)=\D(\Psi_1)+\D(\Psi_2)$
where $\Psi\in \Sp(V=V_1\times V_2,\omega=\omega_1\times \omega_2)$
is given by $\Psi(z_1,z_2)=(\Psi_1z_1,\Psi_2z_2)$ where $\Psi_i\in
\Sp(V_i, \omega_i).$
\end{enumerate}
The Maslov index of a loop of coisotropic subspaces is given (up to
a sign) as the mean index of a certain path of symplectic matrices.
\begin{Definition}[\emph{Maslov Index for Coisotropic Subspaces}; cf.\cite{Zi}]
\label{def:Maslov_Index} Consider $$\Cc=(\Cc_t)_{t\in[0,1]}$$ an
oriented loop of coisotropic subspaces of $(\reals^{2n}, \omega_0)$
and $$H_t \colon \Cc_0 / \Cc_0^{\omega_0} \rightarrow \Cc_t /
\Cc_t^{\omega_0}$$ a path of symplectic linear maps. Recall that a
loop $\Cc$ is oriented if one can orient the space $\Cc_t$
(continuous in $t$) so that $\Cc_0$ and $\Cc_1$ have the same
orientation. Pick a path
\begin{eqnarray}
\label{eqnar:prop_sympl_path} \Psi \colon [0,1] \rightarrow \Sp(2n)
\quad \mbox{satisfying} \quad \Psi_0=Id, \; \Psi_t(\Cc_0)=\Cc_t \;
\mbox{and} \; {\Psi_t}{\Big{|}}_{C_0 / C_0^{\omega}} = H_t
\end{eqnarray}
and define the real valued index $\mu\colon \mathfrak{C} \rightarrow
\reals$ by
$$\mu(\mathcal{C},H):=-\D(\Psi),$$
where $\mathfrak{C}$ is the set of loops of coisotropic subspaces of
$(\reals^{2n}, \omega_0).$

If the loop $\Cc$ is not oriented, we define the Maslov index
$\mu(\Cc,H)$ as half of the Maslov index of the loop obtained by
traversing the initial loop twice.
\end{Definition}
\begin{Proposition}
\label{prop:Maslov_welldefined} The Maslov index given in
Definition~\ref{def:Maslov_Index} is well defined.
\end{Proposition}

\begin{proof}
We prove this proposition in three steps by considering the
following cases:
\begin{enumerate}
\item $\Cc$ is the constant loop where $\Cc_t=L_0$ is a fixed Lagrangian subspace of $(\reals^{2n},\omega_0)$
\item $\Cc$ is the constant loop where $\Cc_t=\Cc_0$ is a fixed coisotropic subspace of $(\reals^{2n},\omega_0)$
\item General case: $\Cc$ is a loop of coisotropic subspaces of $(\reals^{2n},
\omega_0)$.\\
\end{enumerate}

Step $1$: Assume, without loss of generality, that $\Cc$ is the
constant \emph{horizontal loop} $L_0:=\{ (x,y)\in \reals^{2n}\colon
y=0\}.$ Then consider $\Psi\colon [0,1]\rightarrow \Sp(2n)$ as
in~(\ref{eqnar:prop_sympl_path}) and notice that since $\Cc_t=L_0$
is Lagrangian, $H\equiv0$. For $t\in[0,1],$ we have that $\Psi_t$
fixes the lagrangian $L_0$ if and only if it is of the form
$$
\left( \begin{array}{cc}
A_t & B_t \\
0 & A_t^{-T}\end{array} \right) \mbox{ where }
B_t^TA_t^T=A_t^{-1}B_t.
$$
This path is homotopic to the concatenation of two symplectic paths
of the form:
$$
\Psi^{'}_t=\left( \begin{array}{cc}
\tilde{A_t} & 0 \\
0 & \tilde{A_t}^{-T}\end{array} \right)\quad\mbox{ and
}\quad\Psi^{''}_t=\left( \begin{array}{cc}
\tilde{A_1} & \tilde{B_t} \\
0 & \tilde{A_1}^{-T}\end{array} \right)
$$
where we essentially first travel along $\Psi_t$ with $B_t=0$ and
then, when we reach
$$\left( \begin{array}{cc}
A_1 & B_0=0 \\
0 & A_1^{-T}\end{array} \right), $$ we build up $B_t$ from $0$ to
$B_1.$

Since $\Psi^{''}_t$ has constant eigenvalues, $\D(\Psi^{''})=0$.
Hence, by property~(\ref{propertiesMeanIndex_concatenation}), the
mean index of $\Psi$ is equal to the mean index $\Psi^{'}$.

Suppose that $\tilde{A_t}$ is diagonalizable, i.e., it can be
written in the form
\begin{eqnarray}
\label{assumption_diagonalizable} \tilde{A_t}=P_t \underbrace{\left(
\begin{array}{ccc}
{(A_1)}_t &  & 0 \\
 & \ddots  & \\
 0 & & {(A_n)}_t\end{array} \right)}_{=\colon D_t} (P_t)^{-1}
\end{eqnarray}
where $P_t\in \Sp(2n)$ and each block $(A_j)_t$ corresponds to an
eigenvalue $(\lambda_j)_t$ of $\tilde{A}_t$. Then, in this case,
$$\left( \begin{array}{cc}
\tilde{A_t} & 0 \\
0 & \tilde{A_t}^{-T}\end{array} \right) = \left( \begin{array}{cc}
P_t & 0 \\
0 & P_t^{-T}\end{array} \right) \underbrace{\left( \begin{array}{cc}
D_t & 0 \\
0 & D_t^{-T}\end{array} \right)}_{=\colon\Gamma_t} \left(
\begin{array}{cc}
P_t & 0 \\
0 & P_t^{-T}\end{array} \right)^{-1}$$ and, by the naturality
property of the map $\rho,$ we have
$\rho(\Psi_t^{'})=\rho(\Gamma_t)$ for all $t\in [0,1].$
\begin{Claim}
For all $t\in[0,1],$ we have $\rho(\Gamma_t)=1$.
\end{Claim}

\begin{proof}
For the sake of simplicity, we will drop, for now, the subscript $t$
in the notation. By Remark~\ref{rmk:def_rho}, we have
\begin{eqnarray}\label{eqn:rho}
\rho(\Gamma)&:=& (-1)^{m_0} \hspace*{-0.5cm}\displaystyle\prod_{\tiny{\lambda \in \sigma(\Gamma)\cap S^1\backslash\{-1,1\}}} \hspace*{-0.7cm} \lambda^{m_+{(\lambda)}}\nonumber \\
&=&(-1)^{m_0} \hspace*{-0.5cm} \displaystyle\prod_{\substack{\lambda \in \sigma(\Gamma)\cap S^1\backslash\{-1,1\}\\ \footnotesize{\mbox{Im} \lambda >0}}} \hspace*{-0.7cm}\lambda^{m_+{(\lambda)}}\; \bar{\lambda}^{m_+{(\bar{\lambda})}} \nonumber\\
&=&(-1)^{m_0} \displaystyle\prod_{\substack{\tiny{\lambda \in
\sigma(\Gamma)\cap S^1\backslash\{-1,1\}} \\ \tiny{\mbox{Im} \lambda
>0}}}  \lambda^{m_+{(\lambda)}-m_+{(\bar{\lambda})}}
\end{eqnarray}
where $\sigma(\Gamma)$ is the spectrum of $\Gamma.$ Recall that only
the eigenvalues of $\Gamma$ on the unit circle and on the negative
real axis contribute to $\rho(\Gamma)$. Regarding the eigenvalues on
$S^1,$ it can be proved, directly from the definition of $m_+,$ that
$m_+(\lambda)=m_+(\overline{\lambda}).$ Hence, using the notation
with the subscript $t$, we obtain by~(\ref{eqn:rho}) that
$\rho(\Gamma_t)=(-1)^{(m_0)_t},$  for each $t\in [0,1],$ where
$$(m_0)_t:= \# \{\{\lambda_t, \lambda_t^{-1}\} \in \sigma(\Gamma_t)\colon
\lambda_t \in \mathbb{R}^-\} = \# \{\lambda_t \in \sigma(D_t)\colon
\lambda_t \in \mathbb{R}^-\}.$$ The last equality follows from the
fact that $\lambda_t$ is an eigenvalue of $D_t$ if and only if
$\lambda_t$ and $\lambda_t^{-1}$ are eigenvalues of $\Gamma_t.$
Since $D_t$ is continuous in $t$ and $\det (D_t) \not=0$, the signs
of $\det(D_0)$ and $\det(D_1)$ are the same. The determinant of
$D_t$ is given by
$$
\det(D_t) = \displaystyle\prod_{\lambda_t\in\reals^-} \lambda_t \;
\underbrace{\displaystyle\prod_{\lambda_t\in\reals^+}
\lambda_t}_{>0} \;\underbrace{\displaystyle\prod_{\lambda_t\in\C
\backslash \reals} \lambda_t}_{>0}
$$
where the products run over $\lambda_t\in\sigma(D_t).$ Then the sign
of $\det(D_t)$ is determined by the number (mod 2) of the real
negative eigenvalues of $D_t$ and we have
$(-1)^{(m_0)_0}=(-1)^{(m_0)_t}$ for all $t\in [0,1].$ Since, by
(\ref{eqnar:prop_sympl_path}) $D_0=Id$ the result follows
immediately.
\end{proof}
Hence, we have proved that, under the
assumption~(\ref{assumption_diagonalizable}), $\rho(\Psi_t^{'})=1$
for a fixed $t\in [0,1].$ Since the set of diagonalizable matrices
is dense in the set of matrices, the result holds for a ``general"
$\Psi_t^{'}.$ It follows that $\D(\Psi^{'})=0$ and hence we have $\D(\Psi)=0.$\\

Step $2$: Consider $\Psi\colon[0,1]\rightarrow \Sp(2n)$ as
in~(\ref{eqnar:prop_sympl_path}) and the symplectic decomposition of
$\reals^{2n}$:
\begin{eqnarray}\label{eqnar:sympldecom}
\reals^{2n}=(\reals^{2n}/\Cc_0 \oplus
\Cc_0^{\omega_0})\oplus\Cc_0/\Cc_0^{\omega_0}.
\end{eqnarray}

Since $\Psi_t\in\Sp(2n),\;\Psi_t(V)=V$ and $\Psi_t(\Cc_0 /
\Cc_0^{\omega_0})=\Cc_0 / \Cc_0^{\omega_0},$ the path $\Psi_t$ has
the form
$$\left[ \begin{array}{cc}
        (\Psi_t)|_V  & 0\\
        0  & H_t
        \end{array} \right]$$
with respect to decomposition~(\ref{eqnar:sympldecom}), where $V
:=\reals^{2n}/\Cc_0\oplus\Cc_0^{\omega_0}.$ By
property~(\ref{propertiesMeanIndex_product}) of the mean index,
$$\D(\Psi)=\D(\Psi |_V)+\D(H).$$
Since $V$ is symplectic and $\Cc_0^{\omega_0}$ is Lagrangian in $V,$
we have by step $1$ that $\D(\Psi|_V)=0$ and hence $\D(\Psi)=\D(H).$
Therefore, the mean index $\D(\Psi)$ only depends on the mean index
of $H$ and the result is
proved for case (2).\\

Step $3$: Let $\Psi\colon[0,1] \rightarrow \Sp(2n)$ be a path as
in~(\ref{eqnar:prop_sympl_path}) and consider a loop
$\Phi\colon[0,1]\rightarrow \Sp(2n)$ which depends only on $\Cc$ and
satisfies $\Phi_t(\Cc_0)=\Cc_t.$ Define the path
$\tilde\Psi\colon[0,1]\rightarrow\Sp(2n)$ by $\tilde\Psi_t
:=\Phi^{-1}_t\Psi_t$ which satisfies $\tilde\Psi_t(\Cc_0)=\Cc_t$ for
all $t\in [0,1].$ By step $2$, $\D(\tilde\Psi)=\D(\tilde{H}),$ where
$\tilde{H_t}\colon \Cc_0 / \Cc_0^{\omega_0} \rightarrow \Cc_t /
\Cc_t^{\omega_0}$ is given by
$$\tilde{H_t}= {\Phi^{-1}_t}{\big{|}_{(\Cc_t / \Cc_t^{\omega_0})}}{\Psi_t}{\big{|}_{(\Cc_0 / \Cc_0^{\omega_0})}}=
 {\Phi^{-1}_t}{\big{|}_{(\Cc_t / \Cc_t^{\omega_0})}}H_t.$$
Since $\Phi$ is a loop, then by
property~(\ref{propertiesMeanIndex_loops}) of the mean index we have
$\D(\tilde\Psi)=\D(\Phi^{-1}\Psi)=\D(\Phi^{-1})+\D(\Psi)$ and
$\D(\tilde{H})=\D\big{(}\Phi^{-1}{\big{|}_{(\Cc_0
/\Cc_0^{\omega})}}\big{)} + \D(H).$ Hence
$$\D(\Psi)=\D(\tilde{H})-\D(\Phi^{-1})$$ which only depends on $H$
and on $\Phi.$ Since $\Phi_t$ only depends on $\Cc_t,\; \D(\Psi)$
only depends on $H$ and $\Cc.$ Therefore, the Maslov index
$\mu(\Cc,H) :=-\D(\Psi)$ depends only on the loop $\Cc=(\Cc_t)$ and
the linear map $H$ and not on the choice of the path $\Psi$ as long
as it satisfies the properties in~(\ref{eqnar:prop_sympl_path}).
\end{proof}

We, now, define the Maslov index of a capped loop lying in a
coisotropic submanifold and tangent to the characteristic foliation
of the coisotropic submanifold.
\begin{Definition}[\emph{Maslov Index of a Capped Loop}]
Let $(W, \omega)$ be a symplectic manifold, $M^{2n-k}$ a coisotropic
submanifold of $(W, \omega)$ and $\FF$ its characteristic foliation.
Consider $x\colon  S^1 \rightarrow M$ a loop in $M$ tangent to $\Ff$
and $u\colon D^2 \rightarrow W$ a capping of the loop $x$ in $W.$ We
have the symplectic vector bundle decomposition
$$
TW\big{|}_{M} = (TW / TM \oplus T\Ff ) \oplus TM/T\Ff.
$$
Assume $x^* T \Ff$ is orientable and hence trivial. Denote by $\xi$
a trivialization of $x^*T\Ff$:
$$x^* T \Ff \stackrel{\xi}{\cong} S^1 \times T_{x(0)} \Ff.$$
Moreover, we have the following isomorphism
$$TW / TM \cong T^* \Ff,$$
and hence $\xi\oplus\xi^*$ can be viewed as a family of symplectic
maps
$$
\Xi_t\colon TW/TM_{x(0)} \oplus T_{x(0)}\Ff \rightarrow TW/TM_{x(t)}
\oplus T_{x(t)}\Ff.
$$
Denote by $H_t\colon (TM / T\Ff)_{x(0)} \rightarrow (TM /
T\Ff)_{x(t)}$ the holonomy along $x.$ The capping $u$ gives rise to
a symplectic trivialization, unique up to homotopy, of $x^* TW.$
Using such a trivialization, the map $\Xi_t \oplus H_t$ can be
viewed as a path
$$\Psi\colon [0,1]\rightarrow \Sp(2n)$$
which, up to some identifications, satisfies
\begin{eqnarray}
\label{properties:sympl_path} \Psi_0=Id,\quad
\Psi_t(T_{x(0)}M)=T_{x(t)}M \;\mbox{ and }\; \Psi_t\big{|}_{(TM /
T\Ff)_{x(0)}}=H_t.
\end{eqnarray}
Define the \emph{Maslov index} of $(x, u)$ as $\mu(x,u)
:=-\D(\Psi)$. If $x^*T\Ff$ is not orientable, we define $\mu(x,u)$
as $ \mu(x^2,u^2)/2$ where $(x^2,u^2)$ is the double cover of
$(x,u)$.
\end{Definition}

\begin{Remark}
By Proposition~\ref{prop:Maslov_welldefined}, $\mu(x, u)$ is
independent of the trivialization $\xi.$ However it may depend on
the capping $u.$ We give some properties of the coisotropic Maslov
index:
\begin{itemize}
\item Homotopy Invariance: $\mu(x, u)$ is invariant under a homotopy of $x$ in a leaf of $\Ff.$
\item Recapping: $\mu(x,u\#A)=\mu(x,u)+2 \left<c_1,A\right>$ where $u\#A$ is the notation for the recapping of $(x,u)$ by a $2$-sphere $A.$
\item Homogeneity: $\mu(x^k,u^k)=k\mu(x, u)$ where $(x^k,u^k)$ is the $k$-fold cover of $(x,u).$
\end{itemize}
\end{Remark}

\end{document}